\documentclass[times]{speauth}

\usepackage{moreverb}

\usepackage{amsfonts,amssymb,amsthm,newlfont,graphicx,subfigure,algorithmic,algorithm,fixltx2e,booktabs,bm}
\usepackage{epstopdf,breqn}
\usepackage{empheq} % autoload amsmath
\usepackage{appendix}
\usepackage[nomarkers,nolists]{endfloat}

\usepackage[colorlinks,bookmarksopen,bookmarksnumbered,citecolor=red,urlcolor=red]{hyperref}
\numberwithin{algorithm}{section}
\newcommand\BibTeX{{\rmfamily B\kern-.05em \textsc{i\kern-.025em b}\kern-.08em
T\kern-.1667em\lower.7ex\hbox{E}\kern-.125emX}}

\newtheorem{thm}{Theorem}[section]

\newtheorem{cor}{Corollary}[section]

\numberwithin{equation}{section}
\renewcommand{\theequation}{\thesection.\arabic{equation}}
\def\simgt{\,\hbox{\lower0.6ex\hbox{$>$}\llap{\raise0.3ex\hbox{$\sim$}}}\,}
\def\simlt{\,\hbox{\lower0.6ex\hbox{$<$}\llap{\raise0.3ex\hbox{$\sim$}}}\,}
\def\simgteq{\,\hbox{\lower0.6ex\hbox{$\ge$}\llap{\raise0.6ex\hbox{$\sim$}}}\,}
\def\simlteq{\,\hbox{\lower0.6ex\hbox{$\le$}\llap{\raise0.6ex\hbox{$\sim$}}}\,}
\makeatletter
\def\user@resume{resume}
\def\user@intermezzo{intermezzo}
\newcounter{previousequation}
\newcounter{lastsubequation}
\newcounter{savedparentequation}
\renewenvironment{subequations}[1][]{%
      \def\user@decides{#1}%
      \setcounter{previousequation}{\value{equation}}%
      \ifx\user@decides\user@resume
           \setcounter{equation}{\value{savedparentequation}}%
      \else
      \ifx\user@decides\user@intermezzo
           \refstepcounter{equation}%
      \else
           \setcounter{lastsubequation}{0}%
           \refstepcounter{equation}%
      \fi\fi
      \protected@edef\theHparentequation{%
          \@ifundefined {theHequation}\theequation \theHequation}%
      \protected@edef\theparentequation{\theequation}%
      \setcounter{parentequation}{\value{equation}}%
      \ifx\user@decides\user@resume
           \setcounter{equation}{\value{lastsubequation}}%
         \else
           \setcounter{equation}{0}%
      \fi
      \def\theequation  {\theparentequation  \alph{equation}}%
      \def\theHequation {\theHparentequation \alph{equation}}%
      \ignorespaces
}{%
  \ifx\user@decides\user@resume
       \setcounter{lastsubequation}{\value{equation}}%
       \setcounter{equation}{\value{previousequation}}%
  \else
  \ifx\user@decides\user@intermezzo
       \setcounter{equation}{\value{parentequation}}%
  \else
       \setcounter{lastsubequation}{\value{equation}}%
       \setcounter{savedparentequation}{\value{parentequation}}%
       \setcounter{equation}{\value{parentequation}}%
  \fi\fi
  \ignorespacesafterend
}
\makeatother
\begin{document}

\runningheads{Kareem T. Elgindy}{A Barycentric Shifted Gegenbauer Pseudospectral Method}

\title{Optimal Control of a Parabolic Distributed Parameter System Using a Barycentric Shifted Gegenbauer Pseudospectral Method}

\author{Kareem T. Elgindy\corrauth}

\address{Mathematics Department, Faculty of Science, Assiut University, Assiut 71516, Egypt}

\corraddr{Mathematics Department, Faculty of Science, Assiut University, Assiut 71516, Egypt}

\begin{abstract}
In this paper, we introduce a novel pseudospectral method for the numerical solution of optimal control problems governed by a parabolic distributed parameter system. The infinite-dimensional optimal control problem is reduced into a finite-dimensional nonlinear programming problem through shifted Gegenbauer quadratures constructed using a stable barycentric representation of Lagrange interpolating polynomials and explicit barycentric weights for the shifted Gegenbauer-Gauss (SGG) points. A rigorous error analysis of the method is presented, and a numerical test example is given to show the accuracy and efficiency of the proposed pseudospectral method.
\end{abstract}


\maketitle

\vspace{-6pt}

\section{Introduction}
\label{int}
Optimal control theory has attracted much attention since the 1950s after the arrival of digital computers, which provided the impetus for the applications of the branch to many complicated problems; cf. \cite{Elgindy2013b}. One of the primary objectives of this significant branch is to find the control signals that will cause a process to satisfy certain physical constraints while optimizing some performance criterion. Analytical methods can solve only fairly simple problems, therefore much research in this area has been devoted to developing accurate and efficient numerical methods to obtain approximate solutions instead of looking for closed form exact solutions that could be very cumbersome or either impossible to determine.

In this paper, we present a novel and powerful numerical method for the solution of an optimal control problem governed by a parabolic distributed parameter system, which has been recently solved numerically by \cite{Rad2014} using radial basis functions. The present method belongs to the class of pseudospectral methods that were largely developed in the 1970s for solving partial differential equations (PDEs), and impetuously imposed itself strongly as `one of the big three technologies for the numerical solution of PDEs' \cite{Trefethen2000}. The proposed pseudospectral method is a strong tool that exhibits exponential convergence rates, and able to produce accurate approximations using a relatively very small number of collocation points. The central idea in this work is to exploit the well-conditioning of numerical integration operators via recasting the optimal control problem into its integral form. We then approximate the involved integral operators by integration matrices based on shifted Gegenbauer quadratures that can be constructed efficiently using the recently developed Gegenbauer quadratures of \cite{Elgindy2015a}. The novel quadratures are defined based on the stable barycentric representation of Lagrange interpolating polynomials and the explicit barycentric weights for the shifted Gegenbauer-Gauss (SGG) points. The pseudospectral method eventually endeavors to reduce the infinite-dimensional optimal control problem to a finite-dimensional nonlinear programming problem with linear constraints that can be solved easily using standard numerical optimization solvers.

The remainder of this paper is structured as follows: In Sections \ref{sec:ps1} and \ref{sec:tifotocp}, we state the mathematical formulation of the optimal control problem and its integral formulation, respectively. In Section \ref{sec:tdiocp}, we present the novel barycentric shifted Gegenbauer pseudospectral method (BSGPM) for discretizing the integral optimal control problem. Section \ref{sec:err} is devoted for a rigorous error and convergence analysis of the proposed method to verify the spectral decay of the error for increasing number of collocation points. A numerical test example is presented in Section \ref{sec:NEAA1} to assess the accuracy and efficiency of the proposed method followed by some concluding remarks in Section \ref{conc}.
\section{Problem Statement}
\label{sec:ps1}
In this study, we are interested in finding the control function $u:D_{L,t_f}^2 \to \mathbb{R}$, and the corresponding state function $x:D_{L,t_f}^2 \to \mathbb{R}$, that minimize the quadratic cost functional,
\begin{equation}
	J(x,u) = \int_0^{{t_f}} {\int_0^L {\left( {{r_1}\,{x^2}(y,t) + {r_2}\,{u^2}(y,t)} \right)dy\,dt} } ,
\end{equation}
subject to the one-dimensional diffusion equation,
\begin{equation}\label{eq:dyneq1}
	{x_t}(y,t) = {x_{yy}}(y,t) + u(y,t),
\end{equation}
with the initial condition,
\begin{equation}\label{eq:init1}
	x(y,0) = f(y),\quad 0 \le y \le L,
\end{equation}
and the boundary conditions,
\begin{align}
	{x_y}(0,t) &= 0,\quad 0 \le t \le {t_f},\label{eq:bound1}\\
	{x_y}(L,t) &= 0,\quad 0 \le t \le {t_f},\label{eq:bound2}
\end{align}
where $D_{L,t_f}^2 = [0,L] \times [0,{t_f}]$, and $L, t_f, {r_1}, {r_2} \in \mathbb{R}^+$.
\section{The Integral Formulation of the Optimal Control Problem}
\label{sec:tifotocp}
Let,
\begin{align}
{I_{q,\tilde y}^{(y)}(\vartheta (y,t))}&{ = \int_0^{\tilde y} {\int_0^{{\sigma _{q - 1}}} { \ldots \int_0^{{\sigma _2}} {\int_0^{{\sigma _1}} {\vartheta  ({\sigma _0},t)\,d{\sigma _0}d{\sigma _1} \ldots d{\sigma _{q - 2}}d\sigma _{q - 1}} } } } ,}\\
{I_{q,\tilde t}^{(t)}(\vartheta (y,t))}&{ = \int_0^{\tilde t} {\int_0^{{\sigma _{q - 1}}} { \ldots \int_0^{{\sigma _2}} {\int_0^{{\sigma _1}} {\vartheta  (y,{\sigma _0})\,d{\sigma _0}d{\sigma _1} \ldots d{\sigma _{q - 2}}d\sigma _{q - 1}} } } } ,}
\end{align}
denote the $q$-fold integrals of any integrable bivariate function $\vartheta(y,t)$ w.r.t. $y$ and $t$, respectively, for any positive real numbers $\tilde y \in [0, L]$ and $\tilde t \in [0, t_f]$. Using the substitution,
\begin{equation}
	{x_{yy}}(y,t) = \phi (y,t),
\end{equation}
for some unknown function $\phi$, we can recover the unknown state function $x$ and its first-order partial derivative $x_y$ in terms of $\phi$ via successive integration and the boundary condition \eqref{eq:bound1} as follows:
\begin{align}
{x_y}(y,t) = I_{1,y}^{(y)}(\phi (y,t)),\\
x(y,t) = I_{2,y}^{(y)}(\phi (y,t)) + {c_1}(t),\label{eq:xyt1}
\end{align}
where ${c_1}(t)$ is some arbitrary function in $t$. The boundary condition \eqref{eq:bound2} yields,
\begin{equation}\label{eq:bound3}
	I_{1,L}^{(y)}(\phi (y,t)) = 0.
\end{equation}
 Now integrating the Dynamical System Eq. \eqref{eq:dyneq1} with respect to $t$ and using the initial condition \eqref{eq:init1} gives,
\begin{equation}\label{eq:xyt2}
	x(y,t) = I_{1,t}^{(t)}(\phi (y,t) + u(y,t)) + f(y).
\end{equation}
Equating Eqs. \eqref{eq:xyt1} and \eqref{eq:xyt2} yields,
\begin{equation}\label{eq:xyt3}
	I_{2,y}^{(y)}(\phi (y,t)) + {c_1}(t) = I_{1,t}^{(t)}(\phi (y,t) + u(y,t)) + f(y).
\end{equation}
Since Eq. \eqref{eq:xyt3} is satisfied for all $(y,t) \in D_{L,t_f}^2, c_1(t)$ can be determined by setting $y = 0$; so
\begin{equation}
	{c_1}(t) = I_{1,t}^{(t)}(\phi (0,t) + u(0,t)) + f(0),
\end{equation}
and the integral one-dimensional diffusion equation can be written as,
\begin{equation}\label{eq:intdynsyser1}
	I_{2,y}^{(y)}(\phi (y,t)) + I_{1,t}^{(t)}(\phi (0,t) - \phi (y,t) + u(0,t) - u(y,t)) = f(y) - f(0).
\end{equation}
The cost functional can also be written as,
\begin{equation}\label{eq:intquadcostfun1}
	J(\phi ,u) = I_{1,{t_f}}^{(t)}I_{1,L}^{(y)}\left( {{r_1}{{\left( {I_{1,t}^{(t)}(\phi (y,t) + u(y,t)) + f(y)} \right)}^2} + {r_2}\,{u^2}(y,t)} \right).
\end{equation}
Hence, the integral optimal control problem is to find the control function $u:D_{L,t_f}^2 \to \mathbb{R}$, and the corresponding second-order derivative of the state function $\phi:D_{L,t_f}^2 \to \mathbb{R}$, w.r.t. $y$, that minimize the quadratic cost functional \eqref{eq:intquadcostfun1} subject to Eqs. \eqref{eq:intdynsyser1} and \eqref{eq:bound3}.
\section{The BSGPM}
\label{sec:tdiocp}
Let $\mathbb{S}_{l,n}^{(\alpha)} = \{x_{l,n,k}^{(\alpha)}, k = 0, \ldots, n\}$, denote the set of the zeroes (SGG nodes) of the $(n+1)$th-degree shifted Gegenbauer polynomial, $G_{l,n+1}^{(\alpha )}(x)$, defined on the interval $[0, l]$, for any $l \in \mathbb{R}^+, n \in \mathbb{Z}^+, \alpha > -1/2$, and let ${\varpi _{l,n,i}^{(\alpha )}}, i = 0, \ldots, n$, be their corresponding Christoffel numbers. We lay a grid of SGG nodes, $\left(y_{L,N_y,i}^{(\alpha)}, t_{t_f,N_t,j}^{(\alpha)}\right), i = 0, \ldots, N_y; j = 0, \ldots, N_t$, on the rectangular domain $D_{L,t_f}^2$, for some $N_y, N_t \in \mathbb{Z}^+$, sorted ascendingly as $y_{L,{N_y},{N_y} + 1}^{(\alpha )} = 0 < y_{L,{N_y},0}^{(\alpha )} < y_{L,{N_y},1}^{(\alpha )} <  \ldots  < y_{L,{N_y},{N_y}}^{(\alpha )} < y_{L,{N_y},{N_y} + 2}^{(\alpha )} = L;\,t_{{t_f},{N_t},{N_t} + 1}^{(\alpha )} = 0 < t_{{t_f},{N_t},0}^{(\alpha )} < t_{{t_f},{N_t},1}^{(\alpha )} <  \ldots  < t_{{t_f},{N_t},{N_t}}^{(\alpha )} < t_{{t_f},{N_t},{N_t} + 2}^{(\alpha )} = {t_f}$, and approximate the function $\phi$ by interpolation at the internal Gauss nodes. For simplicity, let us denote $\phi \left(y_{L,{N_y},s}^{(\alpha )},t_{t_f ,{N_t},k}^{(\alpha )}\right)$, by ${\phi _{s,k}}\, \forall s,k$. The polynomial interpolant of $\phi$ in two dimensions can be written in Lagrange form as follows \cite{Elgindy2016}:
\begin{equation}\label{sec:theshi2:eq:modlag1}
{P_{{N_y},{N_t}}}\phi (y,t) = \sum\limits_{s = 0}^{{N_y}} {\sum\limits_{k = 0}^{{N_t}} {{\phi _{s,k}}\,{}_{L,t_f }\mathcal{L} _{{N_y},{N_t},s,k}^{(\alpha )}(y,t)} } ,
\end{equation}
where ${}_{L,t_f }\mathcal{L} _{{N_y},{N_t},s,k}^{(\alpha )}(y,t),\,s = 0, \ldots ,{N_y};k = 0, \ldots ,{N_t}$, are the bivariate Lagrange interpolating polynomials defined by,
\begin{align}
_{L,{t_f}}{\cal L}_{{N_y},{N_t},s,k}^{(\alpha )}(y,t) &= {}_L{\cal L}_{B,{N_y},s}^{(\alpha )}(y){\mkern 1mu} \,{}_{{t_f}}{\cal L}_{B,{N_t},k}^{(\alpha )}(t)\,{\mkern 1mu} \forall s,k,\label{sec:theshi2:eq:coeff22}\\
{}_L{\cal L}_{B,{N_y},s}^{(\alpha )}(y) &= \frac{{\xi _{L,{N_y},s}^{(\alpha )}}}{{y - y_{L,{N_y},s}^{(\alpha )}}}/\sum\limits_{j = 0}^{{N_y}} {\frac{{\xi _{L,{N_y},j}^{(\alpha )}}}{{y - y_{L,{N_y},j}^{(\alpha )}}}} \,\forall s,\\
{}_{{t_f}}{\cal L}_{B,{N_t},k}^{(\alpha )}(t) &= \frac{{\xi _{{t_f},{N_t},k}^{(\alpha )}}}{{t - t_{{t_f},{N_t},k}^{(\alpha )}}}/\sum\limits_{j = 0}^{{N_t}} {\frac{{\xi _{{t_f},{N_t},j}^{(\alpha )}}}{{t - t_{{t_f},{N_t},j}^{(\alpha )}}}} \,\forall k,
\end{align}
and $\xi _{l,n,i}^{(\alpha )}, i = 0, \ldots, n$, are the shifted barycentric weights defined by,
\begin{equation}
	\xi _{l,n,i}^{(\alpha )} = 2{( - 1)^i}\sqrt {{4^\alpha }{l^{ - 2(1 + \alpha )}}\left( {l - x_{l,n,i}^{(\alpha )}} \right)\,x_{l,n,i}^{(\alpha )}\,\varpi _{l,n,i}^{(\alpha )}} \,\forall l, n; i = 0, \ldots, n.
\end{equation}
Similarly, we can define the polynomial interpolant of $u$ in two dimensions as follows:
\begin{equation}\label{sec:theshi2:eq:modlag1u1}
{P_{{N_y},{N_t}}}u (y,t) = \sum\limits_{s = 0}^{{N_y}} {\sum\limits_{k = 0}^{{N_t}} {{u _{s,k}}\,{}_{L,t_f }\mathcal{L} _{{N_y},{N_t},s,k}^{(\alpha )}(y,t)} } ,
\end{equation}
The barycentric Lagrange interpolation enjoys several advantages such as being scale-invariant and forward stable for Gauss sets of interpolating points, which makes it very efficient in practice; cf. \cite{Elgindy2015a}.

Let ${\mathbf{P}}_{B,n}^{(q)} = \left( {p_{B,n,i,j}^{(q)}} \right),i,j = 0, \ldots ,n$, be the $q$th-order barycentric Gegenbauer integration matrix, for some $q \in \mathbb{Z}^+$, as defined by \cite{Elgindy2015a}, and denote its $i$th-row vector $\left[ {_lp_{B,n,i,0}^{(q)}{,_l}p_{B,n,i,1}^{(q)}, \ldots {,_l}p_{B,n,i,n}^{(q)}} \right]$ by $_l{\mathbf{P}}_{B,n,i}^{(q)}$. Similar to \cite[Eqs. (4.42)]{Elgindy2016}, we can generate the $q$th-order barycentric shifted Gegenbauer integration matrix (BSGIM), ${}_l{\mathbf{P}}_{B,n}^{(q)} = \left( {{}_lp_{B,n,i,j}^{(q)}} \right), i,j = 0, \ldots ,n$, through the following useful relation:
\begin{equation}\label{eq:shifstandrel1}
	{}_l{\mathbf{P}}_{B,n}^{(q)} = {\left( {\frac{l}{2}} \right)^q}{\mkern 1mu} {\mathbf{P}}_{B,n}^{(q)}.
\end{equation}
By construction, we find that
\begin{subequations}
\begin{align}
I_{1,y_{L,{N_y},i}^{(\alpha )}}^{(y)}\,{P_{{N_y},{N_t}}}\phi (y,t_{{t_f},{N_t},j}^{(\alpha )}) &= \sum\limits_{s = 0}^{{N_y}} {{}_Lp_{B,{N_y},i,s}^{(1)}\,{\phi _{s,j}}} ,\\
I_{1,t_{{t_f},{N_t},j}^{(\alpha )}}^{(t)}\,{P_{{N_y},{N_t}}}\phi (y_{L,{N_y},i}^{(\alpha )},t) &= \sum\limits_{k = 0}^{{N_t}} {{}_{{t_f}}p_{B,{N_t},j,k}^{(1)}\,{\phi _{i,k}}} ,
\end{align}
\end{subequations}
for each $i = 0, \ldots ,{N_y};j = 0, \ldots ,{N_t}$, which is also true for ${P_{{N_y},{N_t}}}u$. Let ${\mathbf{1}} \in {\mathbb{R}^{{N_y} + 1}}$, be the all-ones vector, and ``$\otimes$'' denotes the Kronecker product. The discrete integral one-dimensional diffusion equation can be written at the SGG mesh grid as,
\begin{align}
_L{\mathbf{P}}_{B,{N_y},i}^{(2)}&\left( {\phi \left( {{\bm{y}}_{L,{N_y}}^{(\alpha )},t_{{t_f},{N_t},j}^{(\alpha )}} \right)} \right){ + _{{t_f}}}{\mathbf{P}}_{B,{N_t},j}^{(1)}\left( \phi \left( {y_{L,{N_y},{N_y} + 1}^{(\alpha )},{\bm{t}}_{{t_f},{N_t}}^{(\alpha )}} \right) - \phi \left( {y_{L,{N_y},i}^{(\alpha )},{\bm{t}}_{{t_f},{N_t}}^{(\alpha )}} \right) + u\left( {y_{L,{N_y},{N_y} + 1}^{(\alpha )},{\bm{t}}_{{t_f},{N_t}}^{(\alpha )}} \right)\right.\nonumber\\
 &\left.- u\left( {y_{L,{N_y},i}^{(\alpha )},{\bm{t}}_{{t_f},{N_t}}^{(\alpha )}} \right) \right) = {\mathbf{F}},\label{eq:discintdynsyser1}
\end{align}
for $i = 0, \ldots ,{N_y};\,j = 0, \ldots ,{N_t}$, where $\bm{y}_{L,{N_y}}^{(\alpha )} = {\left[ {y_{L,{N_y},0}^{(\alpha )},y_{L,{N_y},1}^{(\alpha )}, \ldots ,y_{L,{N_y},{N_y}}^{(\alpha )}} \right]^T}, \bm{t}_{{t_f},{N_t}}^{(\alpha )} = {\left[ {t_{{t_f},{N_t},0}^{(\alpha )},t_{{t_f},{N_t},1}^{(\alpha )}, \ldots ,t_{{t_f},{N_t},{N_t}}^{(\alpha )}} \right]^T}$;
\begin{equation}
	{\mathbf{F}} = f\left( {{\bm{y}}_{L,{N_y}}^{(\alpha )}} \right) - f\left( {y_{L,{N_y},{N_y} + 1}^{(\alpha )}} \right) \otimes\, {\mathbf{1}}.
\end{equation}

Let $\mathbf{P}_{B,{N_y},{N_y} + 2}^{(1)} = \left[ {p_{B,{N_y},{N_y} + 2,0}^{(1)},p_{B,{N_y},{N_y} + 2,1}^{(1)}, \ldots ,p_{B,{N_y},{N_y} + 2,{N_y}}^{(1)}} \right]$, be the barycentric Gegenbauer integration vector required for approximating definite integrals over the interval $[-1, 1]$ as described by
\cite[Algorithms 6 and 7]{Elgindy2015a}. Similar to Eq. \eqref{eq:shifstandrel1}, we can easily construct the shifted barycentric Gegenbauer integration vector,
${}_L{\mathbf{P}}_{B,{N_y},{N_y} + 2}^{(1)} = \left[ {}_Lp_{B,{N_y},{N_y} + 2,0}^{(1)}\right.$, $\left.{}_Lp_{B,{N_y},{N_y} + 2,1}^{(1)}, \ldots ,{}_Lp_{B,{N_y},{N_y} + 2,{N_y}}^{(1)} \right]$, by the following formula:
\begin{equation}
	{}_L\mathbf{P}_{B,{N_y},{N_y} + 2}^{(1)} = \frac{L}{2}\, \mathbf{P}_{B,{N_y},{N_y} + 2}^{(1)}.
\end{equation}
Hence, the discrete form of the boundary condition \eqref{eq:bound3} is given by,
\begin{equation}\label{eq:discbound3}
	{}_L\mathbf{P}_{B,{N_y},{N_y} + 2}^{(1)}\,\phi \left(\bm{y}_{L,{N_y}}^{(\alpha )},t_{{t_f},{N_t},j}^{(\alpha )}\right) = 0,\quad j = 0, \ldots .{N_t}.
\end{equation}
The discrete cost functional can also be written as,
\begin{equation}
	{J_{{N_y},{N_t}}} = \sum\limits_{l = 0}^{{N_t}} {_{{t_f}}p_{B,{N_t},{N_t} + 2,l}^{(1)}\,\sum\limits_{k = 0}^{{N_y}} {_Lp_{B,{N_y},{N_y} + 2,k}^{(1)}} } \left( {{r_1}\,{x_{k,l}^2} + {r_2}{\mkern 1mu} {u_{k,l}^2}} \right),
\end{equation}
where,
\begin{align}
u_{k,l} &= u\left( {y_{L,{N_y},k}^{(\alpha )},t_{{t_f},{N_t},l}^{(\alpha )}} \right)\, \forall k,l,\\
	x_{k,l} &= x\left( {y_{L,{N_y},k}^{(\alpha )},t_{{t_f},{N_t},l}^{(\alpha )}} \right) \approx \sum\limits_{s = 0}^{{N_t}} {_{{t_f}}p_{B,{N_t},l,s}^{(1)}\left( {\phi_{k,s} + u_{k,s}} \right)}  + f_k\, \forall k,l;\\
	f_k &= f\left( {y_{L,{N_y},k}^{(\alpha )}} \right)\, \forall k.
\end{align}
Now to put the pointwise representation of the discrete dynamics \eqref{eq:discintdynsyser1} and constraints \eqref{eq:discbound3} into a standard matrix system form, we introduce the mapping $n = \text{index}(i, j): n = i + j\, (N_y + 2)$, and set $L' = {N_y} + {N_t} + {N_y}{N_t}, L'' = {N_y} + {N_t}\left( {{N_y} + 2} \right)$,
\begin{align}
	\left( {{{\phi_{i,j} }}} \right) &= \left( {{{\hat \phi }_n}} \right) = \bm{\hat \phi}  \in {\mathbb{R}^{L'' + 2}};\\
	\left( {{{u_{i,j} }}} \right) &= \left( {{{\hat u }_n}} \right) = \bm{\hat u} \in {\mathbb{R}^{L'' + 2}}.
\end{align}
 We also define the elements of the auxiliary matrices ${\mathbf{A}} \in {\mathbb{R}^{(L'' + 1) \times (L'' + 2)}}$ and ${\mathbf{B}} \in {\mathbb{R}^{(L'' + 1) \times (L'' + 2)}}$, and the column vector $\bm{\tilde b}$ by,
\begin{subequations}\label{eq:formmatcons11}
\begin{align}
&{{(\mathbf{A})}_{{\text{index}}\left( {i,j} \right),{\text{index}}\left( {k,j} \right)}} = _Lp_{B,{N_y},i,k}^{(2)},\quad k = 0, \ldots ,{N_y},k \ne i,\\
&{{(\mathbf{A})}_{{\text{index}}\left( {i,j} \right),{\text{index}}\left( {i,k} \right)}} = { - _{{t_f}}}p_{B,{N_t},j,k}^{(1)},\quad k = 0, \ldots ,{N_t},k \ne j,\\
&{{(\mathbf{A})}_{{\text{index}}\left( {i,j} \right),{\text{index}}\left( {{N_y} + 1,k} \right)}} = _{{t_f}}p_{B,{N_t},j,k}^{(1)},\quad k = 0, \ldots ,{N_t},\\
&{{(\mathbf{A})}_{{\text{index}}\left( {i,j} \right),{\text{index}}\left( {i,j} \right)}} = _Lp_{B,{N_y},i,i}^{(2)}{ - _{{t_f}}}p_{B,{N_t},j,j}^{(1)},\\
&{{(\mathbf{B})}_{{\text{index}}\left( {i,j} \right),{\text{index}}\left( {i,k} \right)}} = { - _{{t_f}}}p_{B,{N_t},j,k}^{(1)},\quad k = 0, \ldots ,{N_t},\\
&{{(\mathbf{B})}_{{\text{index}}\left( {i,j} \right),{\text{index}}\left( {{N_y} + 1,k} \right)}} = _{{t_f}}p_{B,{N_t},j,k}^{(1)},\quad k = 0, \ldots ,{N_t},\\
&{\left({\bm{\tilde b}}\right)_n} = {\left( {\mathbf{F}} \right)_i},
\end{align}
\end{subequations}
for $i = 0, \ldots ,{N_y};\,j = 0, \ldots ,{N_t}$. Therefore, we can write Eqs. \eqref{eq:discintdynsyser1} in matrix form as,
\begin{equation}\label{eq:discintdynsyser2}
	{\mathbf{\bar A}}\,\bm{\hat{\phi}}  + {\mathbf{\bar B}}\,{\bm{\hat u}} = \bm{\bar{b}},
\end{equation}
where,
\begin{align}
\left( {{{(\mathbf{A})}_{i,*}}:(i + 1) \not \equiv 0\; \left( {\bmod \left( {{N_y} + 2} \right)} \right) \wedge \,i \in \{0, \ldots ,L''\}} \right) = {\mathbf{\bar A}} \in {\mathbb{R}^{^{(L' + 1) \times (L'' + 2)}}},\\
\left( {{{(\mathbf{B})}_{i,*}}:(i + 1) \not \equiv 0\; \left( {\bmod \left( {{N_y} + 2} \right)} \right) \wedge \,i \in \{0, \ldots ,L''\}} \right) = {\mathbf{\bar B}} \in {\mathbb{R}^{^{(L' + 1) \times (L'' + 2)}}},\\
\left( {{{\left(\bm{\tilde b}\right)}_{i}}:(i + 1) \not \equiv 0\; \left( {\bmod \left( {{N_y} + 2} \right)} \right) \wedge \,i \in \{0, \ldots ,L''\}} \right) = {\bm{\bar {b}}} \in {\mathbb{R}^{^{(L' + 1)}}},
\end{align}
and the asterisk ``*'' denotes the whole range of column indices. Hence, the global collocation matrix $\mathbf{\bar{H}}$ is simply given by,
\begin{equation}
	\mathbf{\bar{H}} = \left[\mathbf{\bar{A}}, \mathbf{\bar{B}}\right],
\end{equation}
where ``[.,.]'' is the usual horizontal matrix concatenation notation. Thus, Eq. \eqref{eq:discintdynsyser2} can be rewritten as,
\begin{equation}\label{eq:discintdynsyser3}
	{\mathbf{\bar{H}}}\,{\bm{Z}} = {\bm{\bar b}},
\end{equation}
where the solution vector ${\bm{Z}} \in {\mathbb{R}^{2\,L'' + 4}}$ is given by,
\begin{equation}
	{\bm{Z}} = {\text{vec}}\left[ {\bm {\hat{\phi}} ,\,{\bm{\hat u}}} \right],
\end{equation}
and ``vec'' denotes the vectorization of a matrix. Moreover, if we define the elements of the column vector $\bm{{\bar \phi }} \in \mathbb{R}^{L'+1}$ by,
\begin{equation}
	{\bar \phi _l} = {\hat \phi _{l + \left\lfloor {l/\left( {{N_y} + 1} \right)} \right\rfloor }},\quad l = 0, \ldots ,L',
\end{equation}
then Eqs. \eqref{eq:discbound3} can also be written in the following useful matrix form,
\begin{equation}\label{eq:discbound4}
	\left( {{\mathbf{I}_{{N_t} + 1}}{ \otimes\, _L}{\mathbf{P}}_{B,{N_y},{N_y} + 2}^{(1)}} \right)\bm{\bar \phi}  = {\bm{0}} \in {\mathbb{R}^{{N_t} + 1}},
\end{equation}
where ${\mathbf{I}_{{N_t} + 1}}$ is the identity matrix of order $N_t + 1$. We can further combine Eqs. \eqref{eq:discintdynsyser3} and \eqref{eq:discbound4} in a single linear system form. To this end, define the index vector $\Lambda$ by,
\begin{equation}
\Lambda  = \left[ {0(1)L'} \right] + \left\lfloor {\frac{1}{{{N_y} + 1}}\left[ {0(1)L'} \right]} \right\rfloor ,
\end{equation}
where $\left[ {0(1)L'} \right] = [0,1, \ldots ,L']$, and ``$\left\lfloor {.} \right\rfloor$'' denotes the floor function. Moreover, let
\begin{equation}
{\mathbf{\Psi }} \in {\mathbb{R}^{({N_t} + 1) \times (2{\kern 1pt} L'' + 4)}}:\left\{ \begin{array}{l}
	{\left( {\mathbf{\Psi }} \right)_{i,\Lambda }} = {\left( {{I_{{N_t} + 1}}{ \otimes\, _L}P_{B,{N_y},{N_y} + 2}^{(1)}} \right)_{i,*}},\quad i = 0, \ldots ,{N_t},\\
	{\left( {\mathbf{\Psi }} \right)_{i,j}} = 0,\quad i = 0, \ldots ,{N_t};j \notin \Lambda .
	\end{array} \right.
\end{equation}
Hence, Eqs. \eqref{eq:discintdynsyser3} and \eqref{eq:discbound4} can be written as,
\begin{equation}\label{eq:finaldyncon1}
{\mathbf{H}}\,{\bm{Z}} = {\bm{b}},\\
\end{equation}
where,
\begin{align}
	{\mathbf{H}} &= [{\mathbf{\bar H}};{\mathbf{\Psi }}],\\
	{\bm{b}} &= {\text{vec}}\left[{\bm{\bar b}},{\bm{0}}\right],
\end{align}
$\bm{0} \in \mathbb{R}^{N_t+1}$, and ``$[.;.]$'' is the vertical matrix concatenation along columns defined by ``${[{.^T},{.^T}]^T}$''. To write the discrete cost functional in terms of the solution vector $\bm{Z}$, we introduce the mapping $m = \text{index}(k, l): m = k + l\, (N_y + 2)$, and the notation,
\begin{equation}
	{(\bm{v})_{(k)}} = \underbrace {\bm{v} \circ \bm{v} \circ  \ldots  \circ \bm{v}}_{k - {\text{times}}}\, \forall \bm{v} \in \mathbb{R}^l, l \in \mathbb{Z}^+,
\end{equation}
where ``$\circ$'' denotes the Hadamard (entrywise) product. Moreover, let $\bm{{{\check \phi }}} = \left({{\check \phi }_i}\right) \in \mathbb{R}^{L'+1}$ and $\bm{{{\check u }}} = \left({{\check u }_i}\right) \in \mathbb{R}^{L'+1}$:
\begin{align}
	{{\check \phi }_i} &= {({\bm{Z})}_{\left\lfloor {\frac{i}{{{N_t} + 1}}} \right\rfloor  + \left( {{N_y} + 2} \right) \cdot \left( {1 - {\delta _{\frac{i}{{{N_t} + 1}},\left\lfloor {\frac{i}{{{N_t} + 1}}} \right\rfloor }}} \right) \cdot i \equiv 0\; \left({\bmod \left( {{N_t} + 1} \right)} \right)}},\quad i = 0, \ldots ,L',\\
	{{\check u}_i} &= {({\bm{Z})}_{L'' + \left\lfloor {\frac{i}{{{N_t} + 1}}} \right\rfloor  + \left( {{N_y} + 2} \right) \cdot \left( {1 - {\delta _{\frac{i}{{{N_t} + 1}},\left\lfloor {\frac{i}{{{N_t} + 1}}} \right\rfloor }}} \right) \cdot i \equiv 0\; \left({\bmod \left( {{N_t} + 1} \right)} \right) + 2}},\quad i = 0, \ldots ,L',
\end{align}
then the sought discrete cost functional can be written as,
\begin{equation}\label{eq:finalcostfun1}
	{J_{{N_y},{N_t}}}{ = _{{t_f}}}{\mathbf{P}}_{B,{N_t},{N_t} + 2}^{(1)}\left( {{I_{{N_t} + 1}}{ \otimes\, _L}{\mathbf{P}}_{B,{N_y},{N_y} + 2}^{(1)}} \right)\left( {{r_1}\,{{\left( {{\bm{\bar x}}} \right)}_{(2)}} + {r_2}\,{{\left( {{\bm{\bar u}}} \right)}_{(2)}}} \right),
\end{equation}
where,
\begin{align}
	{\left( {{\bm{\bar u}}} \right)_q} &= {{\hat u}_{q + \left\lfloor {\frac{q}{{\left( {{N_y} + 1} \right)}}} \right\rfloor }} = {({\bm{Z})}_{L'' + q + \left\lfloor {\frac{q}{{\left( {{N_y} + 1} \right)}}} \right\rfloor  + 2}},\quad \;q = 0, \ldots ,L',\\
	{\left( {{\bm{\bar x}}} \right)_q} &= {{\hat x}_{q + \left\lfloor {\frac{q}{{\left( {{N_y} + 1} \right)}}} \right\rfloor }},\quad q = 0, \ldots ,L',\\
	{{\hat x}_m} &= {x_{k,l}} = {\left[ {{I_{{N_y} + 1}}{ \otimes\, _{{t_f}}}{\mathbf{P}}_{B,{N_t},l}^{(1)}} \right]_{k,*}}\left( {\bm{\check \phi}  + \bm{\check u}} \right) + {f_k},\quad m = 0, \ldots ,L'' + 1.
\end{align}
Hence, the optimal control problem has been reduced to a nonlinear programming problem, in which we seek the minimization of the objective function ${J_{{N_y},{N_t}}}$ defined by \eqref{eq:finalcostfun1} subject to the linear constraints given by \eqref{eq:finaldyncon1}. Solving for $\bm{Z}$ yields the values of the functions $\phi$ and $u$ at the solution nodes $\left( {y_{L,{N_y},i}^{(\alpha )},t_{{t_f},{N_t},j}^{(\alpha )}} \right),i = 0, \ldots ,{N_y} + 1;j = 0, \ldots ,{N_t}$. To recover the state function $x$ at those nodes, we can use Eq. \eqref{eq:xyt2} to obtain,
\begin{equation}
	{x_{i,j}} \approx\, _{{t_f}}{\mathbf{P}}_{B,{N_t},j}^{(1)}\left( {\phi \left( {y_{L,{N_y},i}^{(\alpha )},\bm{t}_{{t_f},{N_t}}^{(\alpha )}} \right) + u\left( {y_{L,{N_y},i}^{(\alpha )},\bm{t}_{{t_f},{N_t}}^{(\alpha )}} \right)} \right) + {f_i}\,\forall i,j.
\end{equation}
Furthermore, we can generate the approximation of the state profile on $D_{L,t_f}^2$ using the polynomial interpolant of $x$ given by,
\begin{equation}
{P_{{N_y},{N_t}}}x(y,t) = \sum\limits_{s = 0}^{{N_y}} {\sum\limits_{k = 0}^{{N_t}} {{x_{s,k}}\,{}_{L,t_f }\mathcal{L} _{{N_y},{N_t},s,k}^{(\alpha )}(y,t)} }.
\end{equation}
%We shall refer to our novel numerical scheme by the barycentric shifted Gegenbauer pseudospectral method (BSGPM).
%
\section{Error Analysis of the BSGPM}
\label{sec:err}
Let
\begin{equation}
	I_{q,\tilde x}^{(x)}(f(x)) = \int_0^{\tilde x} {\int_0^{{\sigma _{q - 1}}} { \ldots \int_0^{{\sigma _2}} {\int_0^{{\sigma _1}} {f({\sigma _0}){\mkern 1mu} d{\sigma _0}d{\sigma _1} \ldots d{\sigma _{q - 2}}d{\sigma _{q - 1}}} } } } ,
\end{equation}
be the $q$-fold integral of any integrable single-variable function $f(x)$, for some positive real number $\tilde x \in [0, l], l \in \mathbb{R}^+$, and denote $\mathbb{Z}^+ \cup \{0\}$ by $\mathbb{Z}_0^+$. The following two theorems highlight the truncation error and the error bounds of the barycentric shifted Gegenbauer quadrature (BSGQ) associated with the BSGIM, $_l{\mathbf{P}}_{B,n}^{(1)}\, \forall n \in \mathbb{Z}^+$. The proof of both theorems follow that of \cite[Theorems 4.1 \& 4.3]{Elgindy2016}.
\begin{thm}\label{sec:err:thm1}
Let $f(x) \in C^{n + 1}[0, l]$, be interpolated by the shifted Gegenbauer polynomials at the SGG nodes, $x_{l,n,i}^{(\alpha )} \in \mathbb{S}_{l,n}^{(\alpha )},\,i = 0, \ldots ,n$, then there exist some numbers $\zeta _{l,n,i}^{(\alpha )} = \zeta \left(x_{l,n,i}^{(\alpha )}\right) \in (0,l),\,i = 0, \ldots ,n$, such that,
\begin{equation}\label{sec:err:eq:squadki1}
I_{1,x_{l,n,i}^{(\alpha )}}^{(x)}(f(x)) =\, _l{\mathbf{P}}_{B,n,i}^{(1)}\;{\bm{f}} + E_{1,l,n}^{(\alpha )}\left( {x_{l,n,i}^{(\alpha )},\zeta _{l,n,i}^{(\alpha )}} \right),
\end{equation}
where ${\bm{f}} = f\left( {{\mathbf{x}}_{l,n}^{(\alpha )}} \right),\;{\mathbf{x}}_{l,n}^{(\alpha )} = {\left[ {x_{l,n,0}^{(\alpha )},x_{l,n,1}^{(\alpha )}, \ldots ,x_{l,n,n}^{(\alpha )}} \right]^T}$,
\begin{equation}\label{sec1:eq:errorkimohat}
%E_{1,l,n}^{(\alpha )}\left( {x_{l,n,i}^{(\alpha )},\zeta _{l,n,i}^{(\alpha )}} \right) = {\left( {\frac{l}{2}} \right)^{n + 1}}\frac{{{f^{(n + 1)}}\left(\zeta _{l,n,i}^{(\alpha )}\right)}}{{(n + 1)!{\mkern 1mu} K_{n + 1}^{(\alpha )}}}I_{1,x_{l,n,i}^{(\alpha )}}^{(x)}\left( {G_{l,n + 1}^{(\alpha )}(x)} \right),
E_{1,l,n}^{(\alpha )}\left( {x_{l,n,i}^{(\alpha )},\zeta _{l,n,i}^{(\alpha )}} \right) = \frac{{{f^{(n + 1)}}\left( {\zeta _{l,n,i}^{(\alpha )}} \right)}}{{(n + 1)!\,K_{l,n + 1}^{(\alpha )}}}I_{1,x_{l,n,i}^{(\alpha )}}^{(x)}\left( {G_{l,n + 1}^{(\alpha )}(x)} \right),
\end{equation}
and $K_{l,n}^{(\alpha )}$ is the leading coefficient of the shifted Gegenbauer polynomial ${G_{l,n}^{(\alpha )}(x)}$.
\end{thm}
\begin{thm}\label{sec1:thm:krooma1}
Assume that $f(x) \in C^{n+1}[0, l]$, and
\[{\left\| {{f^{(n + 1)}}} \right\|_{{L^\infty }[0,l]}} = \mathop {{\sup}}\limits_{0 \le x \le l} \left| {f(x)} \right| \le A \in {\mathbb{R}^ + },\]
for some number $n \in \mathbb{Z}_0^+$. Moreover, let $I_{1,x_{l,n,i}^{(\alpha )}}^{(x)}(f(x))$, be approximated by the BSGQ, for each integration node ${x_{l,n,i}^{(\alpha )}} \in \mathbb{S}_{l,n}^{(\alpha )},\,i = 0, \ldots ,n$. Then there exist some positive constants $D_1^{(\alpha)}$ and $D_2^{(\alpha)}$, independent of $n$ such that the truncation error of the BSGQ, $E_{1,l,n}^{(\alpha )}$, is bounded by the following inequalities:
\begin{equation}
	\left| {E_{1,l,n}^{(\alpha)}} \right| \le \frac{{A{2^{ - 2n - 1}}\Gamma \left( {\alpha  + 1} \right){x_i}{l^{n + 1}}\Gamma \left( {n + 2\alpha  + 1} \right)}}{{\Gamma \left( {2\alpha  + 1} \right)\Gamma \left( {n + 2} \right)\Gamma \left( {n + \alpha  + 1} \right)}}\left( {\left\{ \begin{array}{l}
	1,\quad n \ge 0 \wedge \alpha  \ge 0,\\
	\frac{{\left( {n + 1} \right)!\Gamma \left( {2\alpha } \right)}}{{\Gamma \left( {n + 2\alpha  + 1} \right)}}\,\left| {\left( {\begin{array}{*{20}{c}}
{\frac{{n + 1}}{2} + \alpha  - 1}\\
{\frac{{n + 1}}{2}}
\end{array}} \right)} \right|,\quad \frac{{n + 1}}{2} \in \mathbb{Z}^+  \wedge  - \frac{1}{2} < \alpha  < 0
	\end{array} \right.} \right),
\end{equation}
\begin{equation}
	\left| {E_{1,l,n}^{(\alpha)}} \right| < \frac{{A{2^{ - 2n - 1}}\Gamma \left( \alpha  \right)\left| \alpha  \right|{\mkern 1mu} {x_i}{l^{n + 1}}}}{{\sqrt {\left( {n + 1} \right)\left( {2\alpha  + n + 1} \right)} {\mkern 1mu} \Gamma \left( {n + \alpha  + 1} \right)}}\,\left| {\left( {\begin{array}{*{20}{c}}
{\frac{n}{2} + \alpha }\\
{\frac{n}{2}}
\end{array}} \right)} \right|,\quad \frac{n}{2} \in \mathbb{Z}_0^ + \wedge  - \frac{1}{2} < \alpha  < 0,
\end{equation}
\begin{subequations}\label{sec1:ineq:errorkimo1}
	\begin{empheq}[left={\left| {E_{1,l,n}^{(\alpha)}} \right| \le}\empheqlbrace]{align}
	&B_1^{(\alpha)}\frac{{{e^n}\,{l^{n + 1}}\,{x_i}}}{{{2^{2n + 1}}\,{n^{n + 3/2 - \alpha}}}},\quad \alpha \ge 0 \wedge n \gg 1,\\
	&B_2^{(\alpha)}\frac{{{e^n}\,{l^{n + 1}}\,{x_i}}}{{{2^{2n + 1}}\,{n^{n + 3/2}}}},\quad - \frac{1}{2} < \alpha < 0  \wedge n \gg 1,
	\end{empheq}
\end{subequations}
for all $i = 0, \ldots, n$, where $B_1^{(\alpha)} = A D_1^{(\alpha)}$, and $B_2^{(\alpha)} = B_1^{(\alpha)} D_2^{(\alpha)}$.
\end{thm}
Theorems \ref{sec:err:thm1} and \ref{sec1:thm:krooma1} show that the truncation error associated with the BSGQ decays exponentially fast for increasing values of $n$ with $\left| {E_{1,l,n}^{(0)}} \right| < \left| {E_{1,l,n}^{(\alpha)}} \right|\, \forall \alpha \neq 0$, as $n \to \infty$.

The following two theorems further generalize Theorems \ref{sec:err:thm1} and \ref{sec1:thm:krooma1} for higher-order BSGQs.
\begin{thm}\label{sec:err:thm3}
Let $f(x) \in C^{n + 1}[0, l]$, be interpolated by the shifted Gegenbauer polynomials at the SGG nodes, $x_{l,n,i}^{(\alpha )} \in \mathbb{S}_{l,n}^{(\alpha )},\,i = 0, \ldots ,n$, then there exist some numbers $\zeta _{l,n,i}^{(\alpha )} = \zeta \left(x_{l,n,i}^{(\alpha )}\right) \in (0,l),\,i = 0, \ldots ,n$, such that,
\begin{align}
I_{q,x_{l,n,i}^{(\alpha )}}^{(x)}(f(x)) &=\, _l{\mathbf{P}}_{B,n,i}^{(q)}\;{\bm{f}} + E_{q,l,n}^{(\alpha )}\left( {x_{l,n,i}^{(\alpha )},\zeta _{l,n,i}^{(\alpha )}} \right),\label{sec:err:eq:squadki13}\\
%E_{1,l,n}^{(\alpha )}\left( {x_{l,n,i}^{(\alpha )},\zeta _{l,n,i}^{(\alpha )}} \right) = {\left( {\frac{l}{2}} \right)^{n + 1}}\frac{{{f^{(n + 1)}}\left(\zeta _{l,n,i}^{(\alpha )}\right)}}{{(n + 1)!{\mkern 1mu} K_{n + 1}^{(\alpha )}}}I_{1,x_{l,n,i}^{(\alpha )}}^{(x)}\left( {G_{l,n + 1}^{(\alpha )}(x)} \right),
E_{q,l,n}^{(\alpha )}\left( {x_{l,n,i}^{(\alpha )},\zeta _{l,n,i}^{(\alpha )}} \right) &= \frac{{1}}{{(q - 1)!\,(n + 1)!{\mkern 1mu} K_{l,n + 1}^{(\alpha )}}}I_{1,x_{l,n,i}^{(\alpha )}}^{(x)}\left( {G_{l,n + 1}^{(\alpha )}(x)} \right){\left[ {\frac{{{d^{n + 1}}}}{{d{x^{n + 1}}}}\left( {{{\left( {x_{l,n,i}^{(\alpha )} - x} \right)}^{q - 1}}f(x)} \right)} \right]_{x = \zeta _{l,n,i}^{(\alpha )}}}.\label{sec1:eq:errorkimohat3}
\end{align}
\end{thm}
\begin{proof}
Using the above notation, we can write Cauchy's formula for repeated integration in the following simple form,
\begin{equation}
	I_{q,x_{l,n,i}^{(\alpha )}}^{(x)}(f(x)) = \frac{1}{{(q - 1)!}}I_{1,x_{l,n,i}^{(\alpha )}}^{(x)}\left( {{{\left( {x_{l,n,i}^{(\alpha )} - x} \right)}^{q - 1}}f(x)} \right).
\end{equation}
By Theorem \ref{sec:err:thm1},
\begin{equation}
	I_{1,x_{l,n,i}^{(\alpha )}}^{(x)}\left( {{{\left( {x_{l,n,i}^{(\alpha )} - x} \right)}^{q - 1}}f(x)} \right) =\; _l\mathbf{P}_{B,n,i}^{(1)}\,\bm{\bar f} + \frac{{1}}{{(n + 1)!\,K_{l,n + 1}^{(\alpha )}}}I_{1,x_{l,n,i}^{(\alpha )}}^{(x)}\left( {G_{l,n + 1}^{(\alpha )}(x)} \right)\,{\left[ {\frac{{{d^{n + 1}}}}{{d{x^{n + 1}}}}\left( {{{\left( {x_{l,n,i}^{(\alpha )} - x} \right)}^{q - 1}}f(x)} \right)} \right]_{x = \zeta _{l,n,i}^{(\alpha )}}},
\end{equation}
where,
\begin{equation}
	{\bm{\bar f}} = {\left( {x_{l,n,i}^{(\alpha )}{\bm{1}} - {\bm{x}}_{l,n}^{(\alpha )}} \right)_{(q - 1)}} \circ {\bm{f}},
\end{equation}
and ${\mathbf{1}} \in {\mathbb{R}^{n+1}}$, is the all-ones vector. Hence, Theorem \ref{sec:err:thm3} follows directly by noticing that,
\begin{equation}
	\frac{1}{{(q - 1)!}}{\,_l}\mathbf{P}_{B,n,i}^{(1)}\,\left( {{{\left( {x_{l,n,i}^{(\alpha )}{\bm{1}} - {\bm{x}}_{l,n}^{(\alpha )}} \right)}_{(q - 1)}} \circ {\bm{f}}} \right) = \frac{1}{{(q - 1)!}}\,\,\left( {\left( {x_{l,n,i}^{(\alpha )}{\bm{1}} - {\bm{x}}_{l,n}^{(\alpha )}} \right)_{(q - 1)}^T \circ\; _l\mathbf{P}_{B,n,i}^{(1)}} \right)\,{\bm{f}} =\; _l{\mathbf{P}}_{B,n,i}^{(q)}\,{\bm{f}}.
\end{equation}
\end{proof}
\begin{thm}\label{sec1:thm:krooma133}
Assume that $f(x) \in C^{n+1}[0, l]$, and
\begin{equation}
	\mathop {{\sup}}\limits_{0 < x < l} \left| {\frac{{{d^{n + 1}}}}{{d{x^{n + 1}}}}\left( {{{\left( {x_{l,n,i}^{(\alpha )} - x} \right)}^{q - 1}}f(x)} \right)} \right| \le A \in {\mathbb{R}^ + },
\end{equation}
for some numbers $n \in \mathbb{Z}_0^+$, and $q \in \mathbb{Z}^+$. Moreover, let $I_{q,x_{l,n,i}^{(\alpha )}}^{(x)}(f(x))$, be approximated by the $q$th-order BSGQ, for each integration node ${x_{l,n,i}^{(\alpha )}} \in \mathbb{S}_{l,n}^{(\alpha )},\,i = 0, \ldots ,n$. Then there exist some positive constants $D_1^{(\alpha)}$ and $D_2^{(\alpha)}$, independent of $n$ such that the truncation error of the $q$th-order BSGQ, $E_{q,l,n}^{(\alpha )}$, is bounded by the following inequalities:
\begin{equation}
	\left| {E_{q,l,n}^{(\alpha)}} \right| \le \frac{{A{2^{ - 2n - 1}}\Gamma \left( {\alpha  + 1} \right){x_i}{l^{n + 1}}\Gamma \left( {n + 2\alpha  + 1} \right)}}{{(q - 1)!}\,{\Gamma \left( {2\alpha  + 1} \right)\Gamma \left( {n + 2} \right)\Gamma \left( {n + \alpha  + 1} \right)}}\left( {\left\{ \begin{array}{l}
	1,\quad n \ge 0 \wedge \alpha  \ge 0,\\
	\frac{{\left( {n + 1} \right)!\Gamma \left( {2\alpha } \right)}}{{\Gamma \left( {n + 2\alpha  + 1} \right)}}\,\left| {\left( {\begin{array}{*{20}{c}}
{\frac{{n + 1}}{2} + \alpha  - 1}\\
{\frac{{n + 1}}{2}}
\end{array}} \right)} \right|,\quad \frac{{n + 1}}{2} \in \mathbb{Z}^+  \wedge  - \frac{1}{2} < \alpha  < 0
	\end{array} \right.} \right),
\end{equation}
\begin{equation}
	\left| {E_{q,l,n}^{(\alpha)}} \right| < \frac{{A{2^{ - 2n - 1}}\Gamma \left( \alpha  \right)\left| \alpha  \right|{\mkern 1mu} {x_i}{l^{n + 1}}}}{{(q - 1)!}\,{\sqrt {\left( {n + 1} \right)\left( {2\alpha  + n + 1} \right)} {\mkern 1mu} \Gamma \left( {n + \alpha  + 1} \right)}}\,\left| {\left( {\begin{array}{*{20}{c}}
{\frac{n}{2} + \alpha }\\
{\frac{n}{2}}
\end{array}} \right)} \right|,\quad \frac{n}{2} \in \mathbb{Z}_0^ + \wedge  - \frac{1}{2} < \alpha  < 0,
\end{equation}
\begin{subequations}\label{sec1:ineq:errorkimo14}
	\begin{empheq}[left={\left| {E_{q,l,n}^{(\alpha)}} \right| \le}\empheqlbrace]{align}
	&B_1^{(\alpha)}\frac{{{e^n}\,{l^{n + 1}}\,{x_i}}}{{{2^{2n + 1}}\,{n^{n + 3/2 - \alpha}}\,{(q - 1)!}}},\quad \alpha \ge 0 \wedge n \gg 1,\\
	&B_2^{(\alpha)}\frac{{{e^n}\,{l^{n + 1}}\,{x_i}}}{{{2^{2n + 1}}\,{n^{n + 3/2}}\,{(q - 1)!}}},\quad - \frac{1}{2} < \alpha < 0  \wedge n \gg 1,
	\end{empheq}
\end{subequations}
for all $i = 0, \ldots, n$, where $B_1^{(\alpha)} = A D_1^{(\alpha)}$, and $B_2^{(\alpha)} = B_1^{(\alpha)} D_2^{(\alpha)}$.
\end{thm}
\begin{proof}
The proof follows easily from Theorems \ref{sec:err:thm1} and \ref{sec:err:thm3}.
\end{proof}
The following is a direct corollary of Theorem \ref{sec:err:thm3}.
\begin{cor}
Let $f(x) \in C^{n + 1}[0, l]$, be interpolated by the shifted Gegenbauer polynomials at the SGG nodes, $x_{l,n,i}^{(\alpha )} \in \mathbb{S}_{l,n}^{(\alpha )},\,i = 0, \ldots ,n$, and suppose that,
\begin{equation}
	{\left\| {{f^{(k)}}} \right\|_{{L^\infty }[0,l]}} \le {A_{{\max}}} \in {\mathbb{R}^ + }\,\forall k = 0, \ldots ,n + 1.
\end{equation}
Then there exist some numbers $\zeta _{l,n,i}^{(\alpha )} = \zeta \left(x_{l,n,i}^{(\alpha )}\right) \in (0,l),\,i = 0, \ldots ,n$, such that,
\begin{align}
I_{q,x_{l,n,i}^{(\alpha )}}^{(x)}(f(x)) &=\, _l{\mathbf{P}}_{B,n,i}^{(q)}\;{\bm{f}} + E_{q,l,n}^{(\alpha )}\left( {x_{l,n,i}^{(\alpha )},\zeta _{l,n,i}^{(\alpha )}} \right),\label{sec:err:eq:squadki135}\\
%E_{1,l,n}^{(\alpha )}\left( {x_{l,n,i}^{(\alpha )},\zeta _{l,n,i}^{(\alpha )}} \right) = {\left( {\frac{l}{2}} \right)^{n + 1}}\frac{{{f^{(n + 1)}}\left(\zeta _{l,n,i}^{(\alpha )}\right)}}{{(n + 1)!{\mkern 1mu} K_{n + 1}^{(\alpha )}}}I_{1,x_{l,n,i}^{(\alpha )}}^{(x)}\left( {G_{l,n + 1}^{(\alpha )}(x)} \right),
\left| {E_{q,l,n}^{(\alpha )}\left( {x_{l,n,i}^{(\alpha )},\zeta _{l,n,i}^{(\alpha )}} \right)} \right| &\le \frac{{{2^{ - n}}{A_{\max }}\,{l^{1 + n}}{n_{\max,i }}\,x_{l,n,i}^{(\alpha )}\,\Gamma (1 + \alpha )\,\Gamma (1 + n + 2\alpha )}}{{(n + 1)!\,(q - 1)!\,\Gamma (1 + n + \alpha )\,\Gamma (1 + 2\alpha )}} \times \nonumber\\
&\left( {\left\{ \begin{array}{l}
1,\quad \alpha  \ge 0,\\
\frac{{2\,\left| \alpha  \right|\,\left| {\left( \begin{array}{c}
\frac{n}{2} + \alpha \\
\frac{n}{2}
\end{array} \right)} \right|\left( {1 + n} \right)!\,\Gamma (2\alpha )}}{{\sqrt {\left( {1 + n} \right)\left( {1 + n + 2\alpha } \right)} \,\Gamma (1 + n + 2\alpha )}},\quad \frac{n}{2} \in \mathbb{Z}_0^ +  \wedge  - \frac{1}{2} < \alpha  < 0,\\
\frac{{\,\left| {\left( \begin{array}{c}
\frac{{n - 1}}{2} + \alpha \\
\frac{{n + 1}}{2}
\end{array} \right)} \right|\,\left( {n + 1} \right)!\,\Gamma (2\alpha )}}{{\Gamma (1 + n + 2\alpha )}},\quad \frac{{n + 1}}{2} \in {\mathbb{Z}^ + } \wedge  - \frac{1}{2} < \alpha  < 0
\end{array} \right.} \right),\label{sec1:eq:errorkimohat356}
\end{align}
for some positive number $q \in \mathbb{Z}^+$, where,
\begin{equation}
{n_{\max ,i}} = \max \left\{ {(q - 1)!,\mathop {{\sup}}\limits_{0 \le k \le n + 1} \left( {{{(q - 1)}_{n - k + 1}}\,{{\left| {x_{l,n,i}^{(\alpha )} - \zeta _{l,n,i}^{(\alpha )}} \right|}^{q - n + k - 2}}} \right)} \right\},
\end{equation}
and,
\[(x)_n = x(x - 1)...(x - (n - 1)),\]
is the falling factorial.
\end{cor}
\begin{proof}
To simplify the notation, let
\[\mu (x) = \frac{{{d^{n + 1}}}}{{d{x^{n + 1}}}}\left( {{{\left( {x_{l,n,i}^{(\alpha )} - x} \right)}^{q - 1}}f(x)} \right).\]
Then by the general Leibniz rule, we have
\[{\left( {\mu (x)\,f(x)} \right)^{(n + 1)}} = \sum\limits_{k = 0}^{n + 1} {\left( \begin{array}{c}
n + 1\\
k
\end{array} \right){\mu ^{(n - k + 1)}}(x)\,{f^{(k)}}(x)} .\]
Realizing that,
\[\left| {{\mu ^{(n - k + 1)}}(\zeta _{l,n,i}^{(\alpha )})} \right| \le \left\{ \begin{array}{l}
0,\quad n - k > q - 2,\\
(q - 1)!,\quad n - k = q - 2,\\
{(q - 1)_{n - k + 1}}\,{\left| {x_{l,n,i}^{(\alpha )} - \zeta _{l,n,i}^{(\alpha )}} \right|^{q - n + k - 2}},\quad n - k < q - 2.
\end{array} \right\},\]
we find that,
\begin{equation}
	\left|E_{q,l,n}^{(\alpha )}\left( {x_{l,n,i}^{(\alpha )},\zeta _{l,n,i}^{(\alpha )}} \right)\right| \le \frac{{{2^{n + 1}}\,{n_{\max ,i}}\,{A_{\max }}}}{{(q - 1)!{\mkern 1mu} (n + 1)!\,\left|K_{l,n + 1}^{(\alpha )}\right|}}\left|I_{1,x_{l,n,i}^{(\alpha )}}^{(x)}\left( {G_{l,n + 1}^{(\alpha )}(x)} \right)\right|.\label{sec1:eq:errorkimohat35}
\end{equation}
The corollary follows easily using \cite[Lemma 4.1]{Elgindy2016}.
\end{proof}
Using the above error analysis, we can straightforwardly determine the truncation error of the integral dynamical system equation \eqref{eq:intdynsyser1} as stated in the following theorem.
\begin{thm}
Let
\begin{equation}
	\psi (y,t) = \phi (0,t) - \phi (y,t) + u(0,t) - u(y,t)\, \forall (y,t) \in {D_{L,{t_f}}^2}.
\end{equation}
Suppose also that $\psi(y,t) \in C^{N_t + 1}[0, t_f]\, \forall y \in [0,L]$, and,
\begin{equation}
{\left\| {\frac{{{\partial ^k}\phi }}{{\partial {y^k}}}} \right\|_{{L^\infty }\left( {D_{L,{t_f}}^2} \right)}} \le {A_{\max }} \in {\mathbb{R}^ + }{\mkern 1mu} \forall k = 0, \ldots ,{N_y} + 1.
\end{equation}
Then there exist some numbers $\zeta _{L,N_y,i}^{(\alpha ),y} = \zeta \left(y_{L,N_y,i}^{(\alpha )}\right) \in (0,L), \zeta _{t_f,N_t,j}^{(\alpha ),t} = \zeta \left(t_{t_f,N_t,j}^{(\alpha )}\right) \in (0,t_f),\,i = 0, \ldots ,N_y; j = 0, \ldots ,N_t$, such that the BSGPM discretizes the integral dynamical system equation \eqref{eq:intdynsyser1} with a total truncation error, $E_{\text{total},i,j}$ at each SGG point, $\left( {y_{L,{N_y},i}^{(\alpha )},t_{{t_f},{N_t},j}^{(\alpha )}} \right)$, bounded by,
\begin{equation}
\left|{E_{{\text{total},i,j}}}\right| \le D^{(\alpha)} (\varepsilon_{1,j} + \varepsilon_{2,i}),
\end{equation}
where,
\begin{align}
\varepsilon_{1,j} &= \frac{{{4^{ - {N_t}}}\,{B_{max,j}}\,t_f^{{N_t} + 1}\,t_{{t_f},{N_t},j}^{(\alpha )}}}{{({N_t} + 1)!\Gamma \left( {{N_t} + \alpha  + 1} \right)}}\left( {\left\{ \begin{array}{l}
\Gamma \left( {{N_t} + 2\alpha  + 1} \right),\quad \alpha  \ge 0,\\
 - \frac{{2\alpha \left( {\begin{array}{*{20}{c}}
{\frac{{{N_t}}}{2} + \alpha }\\
{\frac{{{N_t}}}{2}}
\end{array}} \right)({N_t} + 1)!\Gamma \left( {2\alpha } \right)}}{{\sqrt {\left( {{N_t} + 1} \right)\left( {{N_t} + 2\alpha  + 1} \right)} }},\quad \frac{{{N_t}}}{2} \in \mathbb{Z}_0^ +  \wedge  - \frac{1}{2} < \alpha  < 0,\\
\left| {\left( {\begin{array}{*{20}{c}}
{\frac{{{N_t} - 1}}{2} + \alpha }\\
{\frac{{{N_t} + 1}}{2}}
\end{array}} \right)} \right|({N_t} + 1)!\Gamma \left( {2\alpha } \right),\quad \frac{{{N_t} + 1}}{2} \in  \mathbb{Z}^+ \wedge  - \frac{1}{2} < \alpha  < 0
\end{array} \right.} \right),\\
\varepsilon_{2,i} &= \frac{{{A_{max}}{n_{max,i}}{2^{1 - {N_y}}}y_{L,{N_y},i}^{(\alpha )}\,{L^{{N_y} + 1}}\,\Gamma \left( {{N_y} + 2\alpha  + 1} \right)}}{{({N_y} + 1)!\,\Gamma (1 + {N_y} + \alpha )}}\left( {\left\{ \begin{array}{l}
1,\quad \alpha  \ge 0,\\
\frac{{2\left| \alpha  \right|\left| {\left( {\begin{array}{*{20}{c}}
{\frac{{{N_y}}}{2} + \alpha }\\
{\frac{{{N_y}}}{2}}
\end{array}} \right)} \right|\left( {{N_y} + 1} \right)!\Gamma \left( {2\alpha } \right)}}{{\sqrt {\left( {{N_y} + 1} \right)\left( {{N_y} + 2\alpha  + 1} \right)} \Gamma \left( {{N_y} + 2\alpha  + 1} \right)}},\,\,\frac{{{N_y}}}{2} \in \mathbb{Z}_0^ +  \wedge  - \frac{1}{2} < \alpha  < 0,\\
\frac{{\left| {\left( {\begin{array}{*{20}{c}}
{\frac{{{N_y} - 1}}{2} + \alpha}\\
{\frac{{{N_y} + 1}}{2}}
\end{array}} \right)} \right|\left( {{N_y} + 1} \right)!\Gamma \left( {2\alpha } \right)}}{{\Gamma \left( {{N_y} + 2\alpha  + 1} \right)}},\,\,\frac{{{N_y} + 1}}{2} \in {\mathbb{Z}^ + } \wedge \frac{{ - 1}}{2} < \alpha  < 0
\end{array} \right.} \right),
\end{align}
\begin{align}
		{n_{\max ,i}} &= \max \left\{ {1,\sup \,\left| {y_{L,{N_y},i}^{(\alpha )} - \zeta _{L,{N_y},i}^{(\alpha ),y}} \right|} \right\},\\
		D^{(\alpha)} &= \frac{{\Gamma (1 + \alpha )}}{{2\,\Gamma (1 + 2\alpha )}},
\end{align}
assuming that,
\begin{equation}
	{\left\| {{{\left[ {\frac{{{\partial ^{{N_t} + 1}}}}{{\partial {t^{{N_t} + 1}}}}\psi (y,t)} \right]}_{t = \zeta _{{t_f},{N_t},j}^{(\alpha ),t}}}} \right\|_{{L^\infty }[0,L]}} \le {B_{\max,j }}\in \mathbb{R}^+\, \forall j.
\end{equation}
\end{thm}
\begin{proof}
A straightforward error analysis shows that,
\begin{align}
\left| {{E_{{\text{total}},i,j}}} \right| \le \frac{1}{{({N_t} + 1)!\,\left| {K_{{t_f},{N_t} + 1}^{(\alpha )}} \right|}}\left| {I_{1,t_{{t_f},{N_t},j}^{(\alpha )}}^{(t)}\left( {G_{{t_f},{N_t} + 1}^{(\alpha )}(t)} \right){\mkern 1mu} {{\left[ {\frac{{{\partial ^{{N_t} + 1}}}}{{\partial {t^{{N_t} + 1}}}}\psi (y,t)} \right]}_{t = \zeta _{{t_f},{N_t},j}^{(\alpha ),t}}}} \right|\\
 + \frac{{{2^{{N_y} + 1}}{\mkern 1mu} {n_{\max ,i}}{\mkern 1mu} {A_{\max }}}}{{({N_y} + 1)!{\mkern 1mu} \left| {K_{L,{N_y} + 1}^{(\alpha )}} \right|}}\left| {I_{1,y_{L,{N_y},i}^{(\alpha )}}^{(y)}\left( {G_{L,{N_y} + 1}^{(\alpha )}(y)} \right)} \right|,
\end{align}
from which the theorem follows.
\end{proof}
\section{Numerical Example}
\label{sec:NEAA1}
In this section, we report the results of the developed BSGPM on the optimal control problem under study with $L = 4, t_f = 1, r_1 = r_2 = 1/2; f(y) = 1 + y$. The BSGPM was applied using $N_y = N_t = N = 4, 5, \ldots, 12$, and $\alpha = -0.4, -0.3, \ldots, 0.9$. The nonlinear programming problem was solved using MATLAB ``fmincon'' constrained optimization solver with the default ``TolFun'' and ``TolCon'' of $10^{-6}$ . The numerical experiments were conducted on a personal laptop equipped with an Intel(R) Core(TM) i7-2670QM CPU with 2.20GHz speed running on a Windows 10 64-bit operating system and provided with MATLAB R2014b (8.4.0.150421) Software. Figure \ref{fig:semilogscale1} shows the plots of the approximate optimal cost functional $J_{N,N}$, the feasibility\footnote{By the feasibility of the solution, we mean the maximum constraint violation.} of the optimal solution $\bm{Z}^*$ as reported by the solver, the maximum error in the initial condition \eqref{eq:init1}, ${\psi}_1$, at the $101$ linearly spaced nodes in the $y$- and $t$-directions from $0$ to $4$, and $0$ to $1$, respectively depicted in semi-logarithmic scale, and the maximum error in the boundary condition \eqref{eq:bound3}, ${\psi}_2$. As can be observed from the figure, the approximate optimal cost functional $J_{N,N}^*$ is approximately $15$ for all input data with feasibility and $\psi_2$ near the machine epsilon. We observe also that discretizations at the SGG points for non-positive $\alpha$-values yield the minimum $\psi_1$-values for small values of $N$, whereas the accuracy degrades for increasing values of $\alpha$-- a result that is consistent with the work of \cite{Elgindy2016} on second-order one-dimensional hyperbolic telegraph equations. We expect also to obtain the optimal approximations in the maximum norm for large values of $N$ through discretizations at the shifted Chebyshev-Gauss points as discussed earlier in Section \ref{sec:err}. Fortunately, the present numerical scheme converges exponentially fast for sufficiently smooth solutions using relatively small number of grids. A sketch of the calculated state and control profiles using $N = 12$ and $\alpha = -0.2$ is shown in Figure \ref{fig:fig1}. Figure \ref{fig:fig2} shows also their profiles at the midpoint $y = 2$. In comparison with \cite{Rad2014} who solved the optimal control problem using $120$ nodal points in both directions, the BSGPM exhibits exponential convergence rates and produces excellent approximations using as small as $5$ nodes in both directions.
\begin{figure}[ht]
\centering
\includegraphics[scale=0.85]{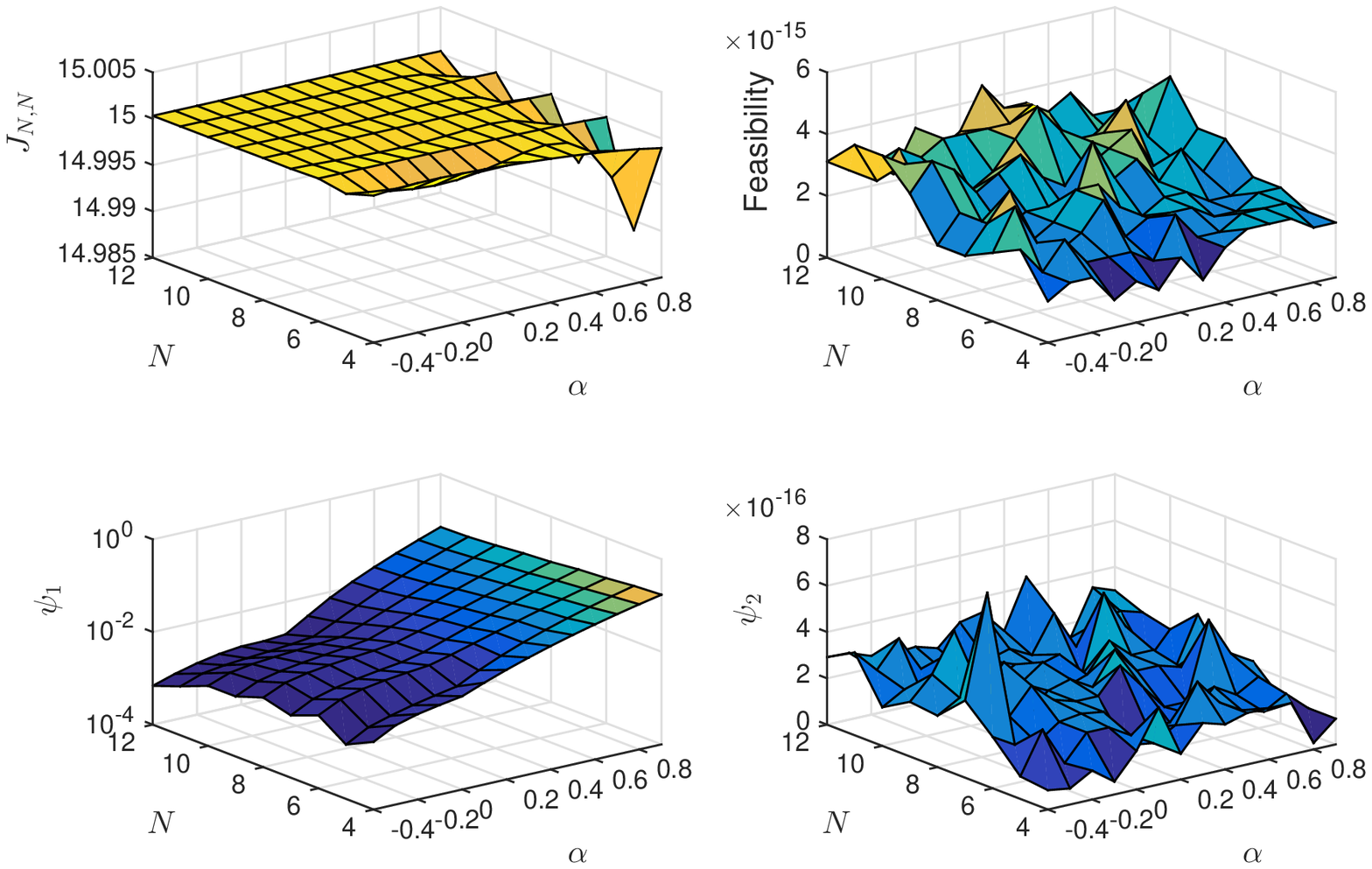}
\caption{The figure shows the plots of the approximate optimal cost functional $J_{N,N}^*$ (upper left), the feasibility of the optimal solution $\bm{Z}^*$ as reported by the solver (upper right), the maximum error in the initial condition \eqref{eq:init1}, ${\psi}_1$, at the $101$ linearly spaced nodes $(x_i, y_i)$ in the $y$- and $t$-directions from $0$ to $4$, and $0$ to $1$, respectively in semi-logarithmic scale (lower left), and the maximum error in the boundary condition \eqref{eq:bound3}, ${\psi}_2$ (lower right). All of the plots were generated using the same $101$ points $(x_i, y_i)$.}
\label{fig:semilogscale1}
\end{figure}
\begin{figure}[ht]
\centering
\includegraphics[scale=1]{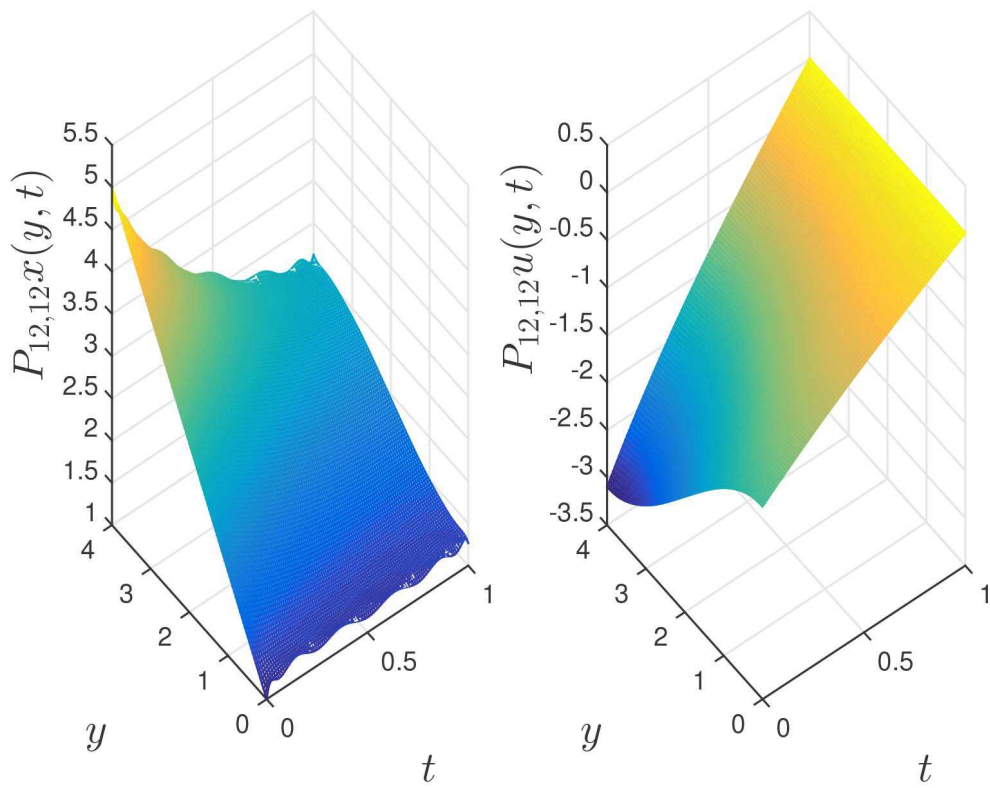}
\caption{The figure shows the state and control profiles on $D_{4,1}^2$ using $N = 12$ and $\alpha = -0.2$.}
\label{fig:fig1}
\end{figure}
\begin{figure}[ht]
\centering
\includegraphics[scale=1]{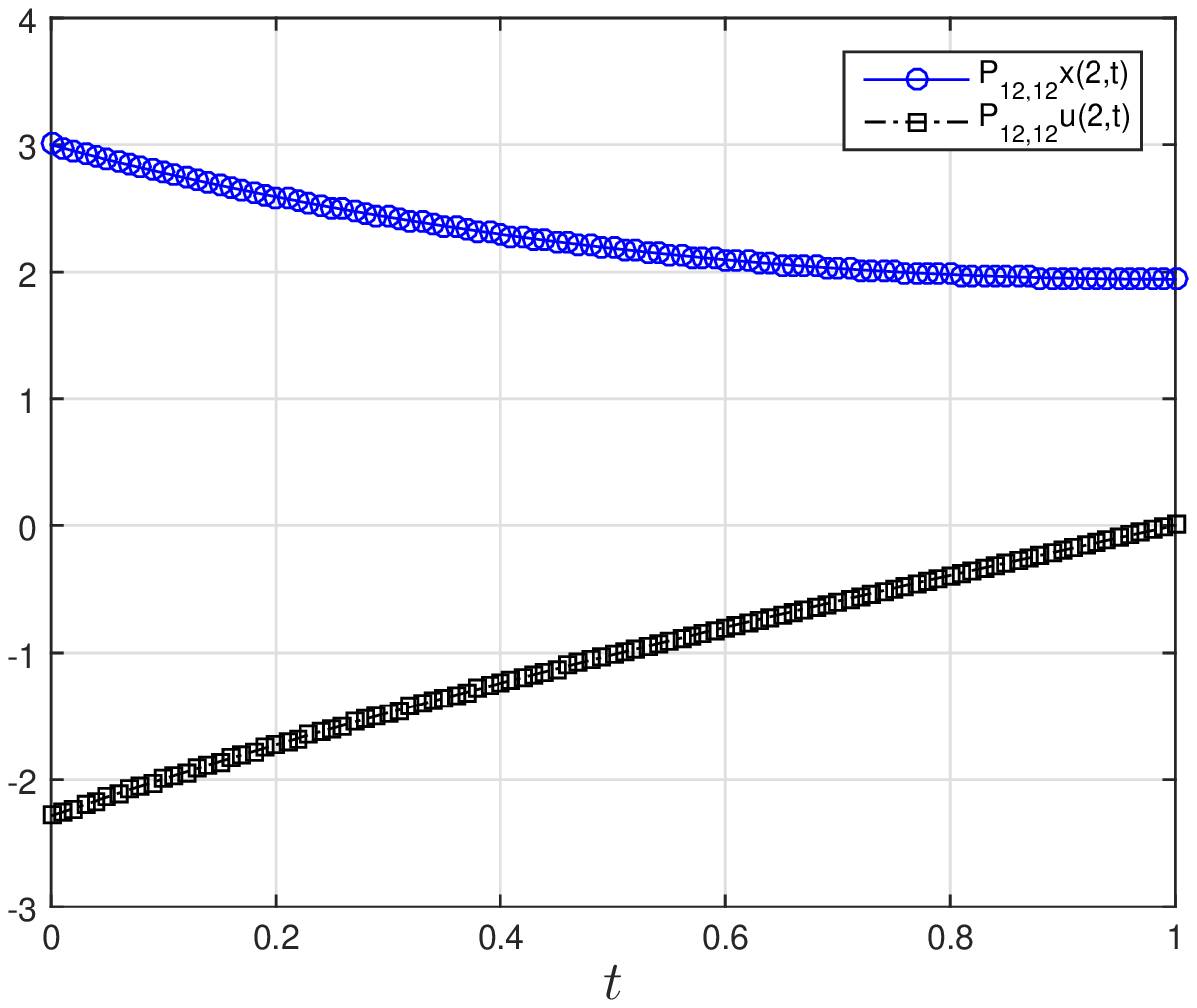}
\caption{The figure shows the state and control profiles at the midpoint $y = 2$ using $N = 12$ and $\alpha = -0.2$.}
\label{fig:fig2}
\end{figure}
\section{Conclusion}
\label{conc}
This paper presented a robust and computationally efficient BSGPM for solving a PDE-governed optimal control problem. A key reason underlying the computationally
streamlined nature of the current approach lies in the accurate discretization of the system dynamics and constraints into a well-conditioned algebraic linear system using stable and high-order BSGQs. Using a practical test example, it is shown that the BSGPM has two significant advantages over the method of \cite{Rad2014}: (i) the method converges exponentially fast, and (ii) the required number of collocation/nodal points to produce high-order approximations is significantly smaller. The test example also suggests that discretizations at the SGG points for non-positive $\alpha$-values yield better approximations for relatively small numbers of expansion terms, whereas the accuracy degrades for increasing values of $\alpha$. The present method provides a strong addition to the arsenal of numerical pseudospectral methods, and can be extended to solve a wide range of PDE-governed optimal control problems arising in numerous applications.
%%
%% ---------------------------------------------
%% References
%%
%%%\bibliographystyle{model1-num-names}
%\bibliographystyle{elsarticle-harv}
%\bibliography{Bib}

%% ---------------------------------------------
%%
\end{document}